\documentclass[reqno, 11pt]{amsart}
\usepackage[margin=1in]{geometry}
\usepackage{amsthm, amsmath, amsfonts, amssymb, mathabx, mathtools}
\usepackage{graphicx, float}
\usepackage{hyperref}
\usepackage{longtable, xcolor}

\newcommand{\R}[0]{\mathbb{R}}

\newcommand{\st}[1]{\substack{#1}}
\newcommand{\mb}[1]{\mathbf{#1}}

\newcommand{\nms}[1]{\| #1 \|}

\DeclareMathOperator*{\argmax}{argmax}

\newtheorem{thm}{Theorem}

\newtheorem{prop}[thm]{Proposition}

\theoremstyle{remark}

\title{An improved example for an autoconvolution inequality} 

\author{Christopher Boyer and Zane Kun Li}
\address{Department of Mathematics, North Carolina State University, Raleigh, NC 27695}
\email{cwboyer@ncsu.edu, zkli@ncsu.edu}

\begin{document}
\begin{abstract}
We give a nonnegative step function with 575 equally spaced intervals such that 
$$\frac{\nms{f \ast f}_{L^{2}(\R)}^{2}}{\nms{f \ast f}_{L^{\infty}(\R)}\nms{f \ast f}_{L^{1}(\R)}} \geq 0.901564.$$
This improves upon a recent result of Deepmind's AlphaEvolve, which found a nonnegative step function with 50 equally space intervals for which the left hand side is $\geq 0.8962$.
Our function was found using simulated annealing and gradient based methods rather than using large language models.
\end{abstract}

\maketitle

\section{Introduction}
In \cite{MartinOBryant}, Martin and O'Bryant studied the following question:
Observe that if $f$ is supported on an interval $I$, we trivially
have $\nms{F}_{L^{\infty}} \geq \nms{F}_{L^{1}}/|I|$ with equality
happening only when $F = c1_{I}$. Intuitively, we do not expect
the autoconvolution $f \ast f$ of a nonnegative function $f$
supported on an interval $I$ to behave like an indicator function.
Can one quantify this?
Indeed, their main result is that if $f \in L^{1}(\R) \cap L^{2}(\R)$, nonnegative and supported on an interval $I$ and $h$ is an even integer, then
$$\nms{f^{\ast h}}_{L^\infty(\R)} > \frac{1.262}{h|I|}\nms{f}_{L^{1}(\R)}^{h}$$
where here $h$ is an integer, $f^{\ast h} = f \ast f^{\ast h - 1}$, and $f^{\ast 2} = f \ast f$.
If we take $h = 2$, $I = [-1/4, 1/4]$, and $\nms{f}_{L^{1}(\R)} = 1$,
then we see that any such nonnegative $f$ supported on $[-1/4, 1/4]$
with $L^{1}$ norm 1 must obey $\nms{f \ast f}_{L^{\infty}(\R)} > 1.262$.
This improved a previous result of Yu \cite{Yu} from 2007. Martin and O'Bryant's
results were subsequently improved to 1.2748 in by Matolcsi and Vinuesa \cite{MatolcsiVinuesa} in 2010 and to 1.28 by Cloninger and Steinerberger \cite{CloningerSteinerberger} in 2017.
We also refer to \cite{BarnardSteinerberger} and \cite{MartinOBryant} for applications
of this and similar autoconvolution and autocorrelation problems to additive combinatorics. See also \cite[Problem 35]{BenGreen100}.

In this short note, we shall concentrate on a different quantification
of the statement that the autoconvolution of a nonnegative function is not an indicator function.
Given a function $F$, from H\"{o}lder's inequality, we have
\begin{align*}
\frac{\nms{F}_{L^{2}(\R)}^{2}}{\nms{F}_{L^{\infty}(\R)}\nms{F}_{L^{1}(\R)}} \leq 1
\end{align*}
with equality attained when $F$ is an indicator function. 
This trivial inequality played a key role in one of the key lemmas
in \cite{MartinOBryant} (namely Lemma 3.2). Since autoconvolution of a nonegative $f$ is intuitively not an indicator functions, it was asked if $F = f \ast f$, can the above bound be improved? 
This question can be rephrased as: what is the value of
\begin{align}\label{target}
c \coloneqq \sup_{\st{f \in L^{1}(\R) \cap L^{2}(\R)\\ f \geq 0}}\frac{\nms{f \ast f}_{L^{2}(\R)}^{2}}{\nms{f \ast f}_{L^{\infty}(\R)}\nms{f \ast f}_{L^{1}(\R)}}?
\end{align}
Trivially, we have $c \leq 1$.
In \cite[Section 5]{MartinOBryant}, it was conjectured that $c = \frac{\log 16}{\pi} \approx 0.88254$ with the optimizer being the function $f = (2x)^{-1/2}1_{0 < x < 1/2}$. 
In \cite{MatolcsiVinuesa}, Matolcsi and Vinuesa found a step function $f$ with 20 steps
giving $c \geq 0.88922$. They however still believed that
$c < 1$.
See also O'Bryant's question on MathOveflow \cite{MathOverflow} and its
responses for
further discussion.
It should be noted if one removes the assumption that $f$ is nonnegative, then it was shown in \cite[Section 5]{deDiosPontMadrid} that the smallest constant is 1 (and so no improvement over H\"{o}lder's inequality is obtained).

In May 2025, AlphaEvolve\footnote{Google Deepmind's AlphaEvolve is an evolutionary coding agent that uses large language models and automated evaluators to generate, test, and evolve algorithms.} \cite{AlphaEvolve} gave a
step function $f$ with 50 steps supported in $[-1/4,1/4]$ yielding $c \geq 0.8962$. 
Our main result is the following improvement over AlphaEvolve's bound (which in fact
we found in November 2024, before we were made aware of AlphaEvolve's example).
Unlike the function found in AlphaEvolve using large language models,
our function was found using gradient based methods.
It would be interesting to see if AlphaEvolve can further improve
lower bounds for $c$.

\begin{thm}\label{main}
We have $c \geq  0.901564$.
That is, there a function $f$ such that
$$\frac{\nms{f \ast f}_{L^{2}(\R)}^{2}}{\nms{f \ast f}_{L^{\infty}(\R)}\nms{f \ast f}_{L^{1}(\R)}} = \frac{96086089410408840289339769}{106577013451431545242354944} \geq 0.901564.$$
More explicitly we take $f = \sum_{n = 0}^{574}v_{n}1_{[n, n + 1)}$ where the $\{v_n\}$ are given in
Appendix \ref{coeff}.
\end{thm}

In Section \ref{setup}, we discretize the problem of estimating \eqref{obj}. In Section
\ref{graphs}, we give graphs of $f$ and $f \ast f$ and compare the result here with that
by Matolcsi and Vinuesa and AlphaEvolve. Finally, in Section \ref{desc}, we describe the algorithmic
method we used to find the example in Theorem \ref{main}.
Python and Mathematica code used for this project
can be found at \url{https://github.com/zkli-math/autoconvolutionHolder}.

\subsection*{Acknowledgements}
ZL is supported by DMS-2409803. We thank Jos\'e Madrid and Stefan Steinerberger
for helpful discussions and encouragement.
The authors would also like to thank Alina Duca for matching ZL
with CB in a call for undergraduate math research projects in Fall 2024.
We are also grateful for the computing resources provided by the North Carolina State University High Performance Computing Services Core Facility (RRID:SCR\_022168). Finally the authors would like to thank the referee
for numerous improvements to the exposition and suggesting a way
to modify our initial function so that it uses integer arithmetic
rather than floating point arithmetic which happened to also give a small boost in ratio.

\section{Reduction to a discrete problem}\label{setup}
We first observe that the raio of $L^p$ norms in \eqref{target} is invariant under translations and dilations.

\begin{prop}[Translation/dilation invariance]\label{tdi}
Let $f_{h}(x) \coloneqq f(x - h)$ and $f^{a}(x) \coloneqq f(ax)$ for some real number $h$ and some $a \neq 0$. Then
$$\frac{\nms{f \ast f}_{L^{2}(\R)}^{2}}{\nms{f \ast f}_{L^{\infty}(\R)}\nms{f \ast f}_{L^{1}(\R)}} = \frac{\nms{f_h \ast f_h}_{L^{2}(\R)}^{2}}{\nms{f_h \ast f_h}_{L^{\infty}(\R)}\nms{f_h \ast f_h}_{L^{1}(\R)}}$$
and
$$\frac{\nms{f \ast f}_{L^{2}(\R)}^{2}}{\nms{f \ast f}_{L^{\infty}(\R)}\nms{f \ast f}_{L^{1}(\R)}} = \frac{\nms{f^{a} \ast f^{a}}_{L^{2}(\R)}^{2}}{\nms{f^{a} \ast f^{a}}_{L^{\infty}(\R)}\nms{f^{a} \ast f^{a}}_{L^{1}(\R)}}.$$
\end{prop}
\begin{proof}
This follows from the observations that
$(f_{h} \ast f_{h})(x) = (f \ast f)(x - 2h)$ and $(f^{a} \ast f^{a})(x) = a^{-1}(f \ast f)(ax)$.
\end{proof}

To discretize the problem of giving a lower bound for \eqref{target},
we will search for an $f$ that is a step function.
Dilation and translation invariance demonstrate that
the particular location and size of support is irrelevant
to problem, what is important is only the heights
of the step function and the number of steps.

Given a sequence of nonnegative real numbers $\{v_n\}_{n = 0}^{N-1}$, let $f = \sum_{n = 0}^{N-1}v_{n}1_{[n, n+1)}$. We compute that $\nms{f \ast f}_{L^{1}(\R)} = \nms{f}_{L^{1}(\R)}^{2} = (\sum_{n = 0}^{N-1}v_{n})^{2}$. Next, we have $f \ast f$ is supported
on $[0, 2N)$ and
\begin{align*}
(f \ast f)(x) = \sum_{n, m = 0}^{N-1}v_{n}v_{m}(1_{[n, n+1)} \ast 1_{[m, m+1)})(x)
\end{align*}
with
\begin{align*}
(1_{[n, n+1)} \ast 1_{[m, m+1)})(x) = \begin{cases}
x - (m + n) & \text{ if } m + n \leq x \leq m + n + 1\\
2+ (m + n) - x & \text{ if } m + n + 1 \leq x \leq m + n + 2\\
0 & \text{ else.}
\end{cases}
\end{align*}
Notice that the graph of $f \ast f$ is just the linear
interpolation of the values $\{(f \ast f)(j)\}_{j= 0}^{2N - 1}$. The easiest way to see this is to observe that the derivative
of $f \ast f$ is constant between its value at integer inputs.
Therefore
\begin{align*}
\nms{f \ast f}_{L^{\infty}(\R)} = \max_{j = 0, 1, \ldots, 2N - 1}(f \ast f)(j).
\end{align*}
Finally, since $f \ast f$ is just the linear interpolation
of the points $(j, (f \ast f)(j))$, we compute
\begin{align*}
\nms{f \ast f}_{L^{2}(\R)}^{2} &= \sum_{j = 0}^{2N - 1}\int_{j}^{j + 1}(f \ast f)^{2}(x)\, dx\\
&= \frac{1}{3}\sum_{j = 0}^{2N - 1}(f \ast f)(j)^{2} + (f \ast f)(j) \cdot (f \ast f)(j + 1) + (f \ast f)(j + 1)^{2}.
\end{align*}
For integers $0 \leq j \leq 2N- 1$, we have
\begin{align}\label{ljdef}
(f \ast f)(j) = \sum_{n, m = 0}^{N - 1}v_{n}v_{m}1_{m + n = j - 1}= \sum_{n = \max(0, j - N)}^{\min(j - 1, N - 1)}v_{n}v_{j - n - 1}.
\end{align}
We denote this expression on the right hand side by $L_j$.
It then follows that for $f = \sum_{n = 0}^{N - 1}v_{n}1_{[n, n + 1)}$, we have
\begin{align}\label{obj}
\frac{\nms{f \ast f}_{L^{2}(\R)}^{2}}{\nms{f \ast f}_{L^{\infty}(\R)}\nms{f \ast f}_{L^{1}(\R)}} = \frac{\frac{1}{3}\sum_{j = 0}^{2N - 1}L_{j}^{2} + L_{j}L_{j + 1} + L_{j + 1}^{2}}{(\max_{0 \leq j \leq 2N - 1}L_j)(\sum_{n = 0}^{N - 1}v_n)^{2}}.
\end{align}
Note that since both $(\sum_{n = 0}^{N - 1}v_n)^{2}$ and $\frac{1}{2}\sum_{j = 0}^{2N - 1}L_{j} + L_{j + 1}$ evaluate to
$\nms{f \ast f}_{L^{1}}$, one can swap out the former for the latter in \eqref{obj} if desired.
Thus our problem is now to find an $N$ and a sequence
of nonnegative real numbers $\{v_n\}_{n = 0}^{N - 1}$ such that \eqref{obj} is as large as possible with $L_j$ given in \eqref{ljdef}.

Given the explicit formula in \eqref{obj} and the fact
that the problem is translation and dilation invariant,
it is easy to check that if we take the sequence of
numbers given in \cite{MatolcsiVinuesa} and \cite{AlphaEvolve} to be $\{v_{n}\}_{n = 0}^{N - 1}$, we obtain 0.889226 and 0.89628, respectively, matching both claims (note that the table in \cite{MatolcsiVinuesa} is to be read first by row, rather than by column and so $v_2 = 0.54399$).
Computing \eqref{obj} on the sequence of coefficients in Appendix \ref{coeff} chosen for Theorem \ref{main} gives 0.901564.

\section{Graphs of $f$ and $f \texorpdfstring{\ast}{*} f$: a Comparison}\label{graphs}
We compare our function from Theorem \ref{main} with Matolcsi and Vinuesa's step function and AlphaEvolve's step function.
Let $\{v_{n}^{MV}\}_{n = 0}^{19}$ be the list of coefficients found in Page 447.e2 of \cite[Appendix A]{MatolcsiVinuesa}, $\{v_{n}^{AE}\}_{n = 0}^{49}$ be the list of coefficients in Section B.2 of \url{https://github.com/google-deepmind/alphaevolve_results/blob/main/mathematical_results.ipynb} but in reverse order (for ease of comparison in the functions defined below),
and $\{v_{n}^{BL}\}_{n = 0}^{574}$ be the list of coefficients in Appendix \ref{coeff}.

We create the functions:
\begin{align*}
f_{1}(x) \coloneqq \sum_{n = 0}^{19}v_{n}^{MV}1_{[n, n + 1)}(x), \quad\quad f_{2}(x) \coloneqq \sum_{n = 0}^{49}v_{n}^{AE}1_{[n, n + 1)}(x), \quad\quad f_{3}(x) \coloneqq \sum_{n = 0}^{574}v_{n}^{BL}1_{[n, n + 1)}(x).
\end{align*}
By Proposition \ref{tdi}, when computing the ratio in \eqref{target}, the location and size
of support of the functions is irrelevant. Thus computing $\nms{f_{i} \ast f_{i}}_{L^{2}(\R)}^{2}/(\nms{f_i \ast f_i}_{L^{\infty}(\R)}\nms{f_i \ast f_i}_{L^{1}(\R)})$ for $i = 1, 2, 3$ will give
0.8962, 0.88922, and 0.901564, respectively.

But in this section we want to compare our function with Matolcsi and Vinuesa's and AlphaEvolve's.
To facilitate this, by Proposition \ref{tdi}, we will shift and rescale $f_i$ to be supported
on $[-1/4, 1/4]$ and normalize the maximum of $f_i \ast f_i$ to be 1.
We define
\begin{align*}
F_{i}(x) \coloneqq \frac{f_{i}(2N_{i}x + \frac{N_i}{2})}{\nms{f_{i}(2N_{i}\cdot + \frac{N_i}{2}) \ast f_{i}(2N_{i}\cdot + \frac{N_i}{2})}_{L^{\infty}(\R)}^{1/2}}
\end{align*}
for $i = 1, 2, 3$ and $N_{1} = 20$, $N_{2} = 50$, and $N_{3} = 575$.
By Proposition \ref{tdi}, $\nms{F_{i} \ast F_{i}}_{L^{2}(\R)}^{2}/(\nms{F_i \ast F_i}_{L^{\infty}(\R)}\nms{F_i \ast F_i}_{L^{1}(\R)}) = \nms{f_{i} \ast f_{i}}_{L^{2}(\R)}^{2}/(\nms{f_i \ast f_i}_{L^{\infty}(\R)}\nms{f_i \ast f_i}_{L^{1}(\R)})$.
We make two comments.
First note that the reflection of $f_{2}(2N_{2}x + N_{2}/2)$ across the vertical axis is in fact the function mentioned in AlphaEvolve's result as it is supported on $[-1/4, 1/4]$ and has heights $\{v_{n}^{AE}\}_{n = 0}^{49}$.
Second, we note that by how $F_{i}$ is defined, $\nms{F_{i} \ast F_{i}}_{L^{\infty}} = 1$.
Thus by plotting $F_{i}$ and $F_{i} \ast F_{i}$ for $i = 1, 2, 3$, we can compare
our result with that of Matolcsi and Vinuesa's and AlphaEvolve's.

In the next three figures, we plot the $F_i$, their autoconvolution $F_{i} \ast F_{i}$, and compare all three autoconvolutions. The $F_{i}$ are all supported on $[-1/4, 1/4]$
and consist of $20, 50$, and $575$ equally spaced intervals while the $F_{i} \ast F_{i}$ are piecewise
linear and supported on $[-1/2, 1/2]$.
\begin{figure}[H]
\centering
\includegraphics[width=\textwidth,keepaspectratio]{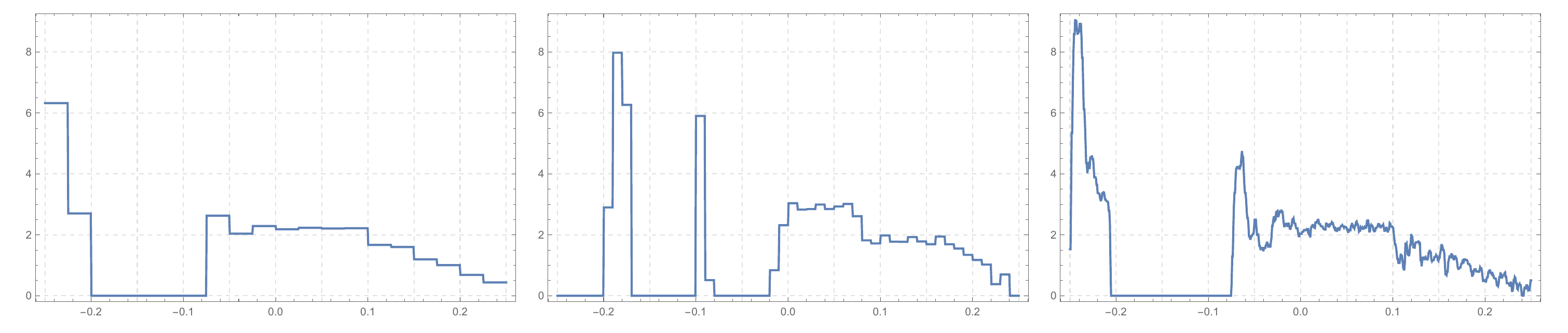}
\caption{Normalized versions of Matolcsi and Vinuesa's step function (left), AlphaEvolve's step function (center), and our step function (right)}
\end{figure}

\begin{figure}[H]
\centering
\includegraphics[width=\textwidth,keepaspectratio]{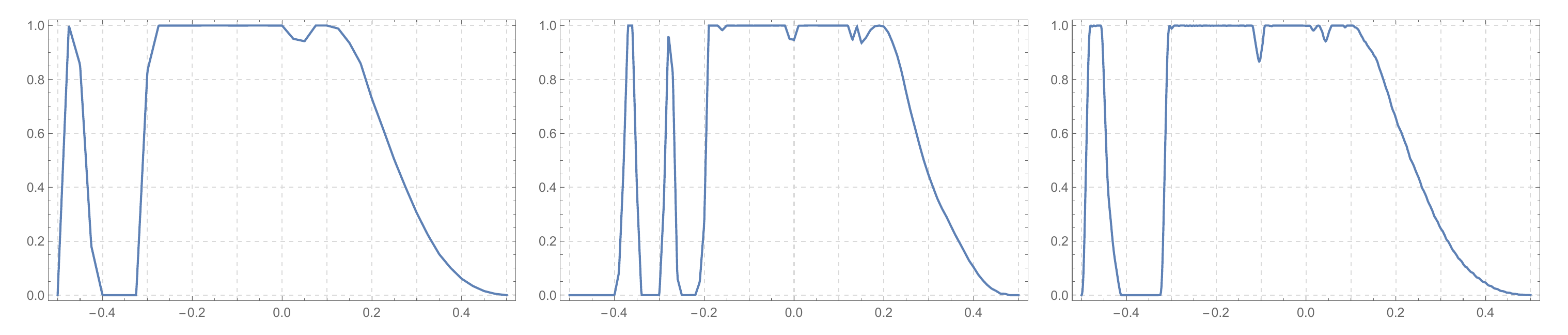}
\caption{$F_1 \ast F_1$ (left), $F_2 \ast F_2$ (center), and $F_3 \ast F_3$ (right)}
\end{figure}

\begin{figure}[H]
\centering
\includegraphics[width=\textwidth,keepaspectratio]{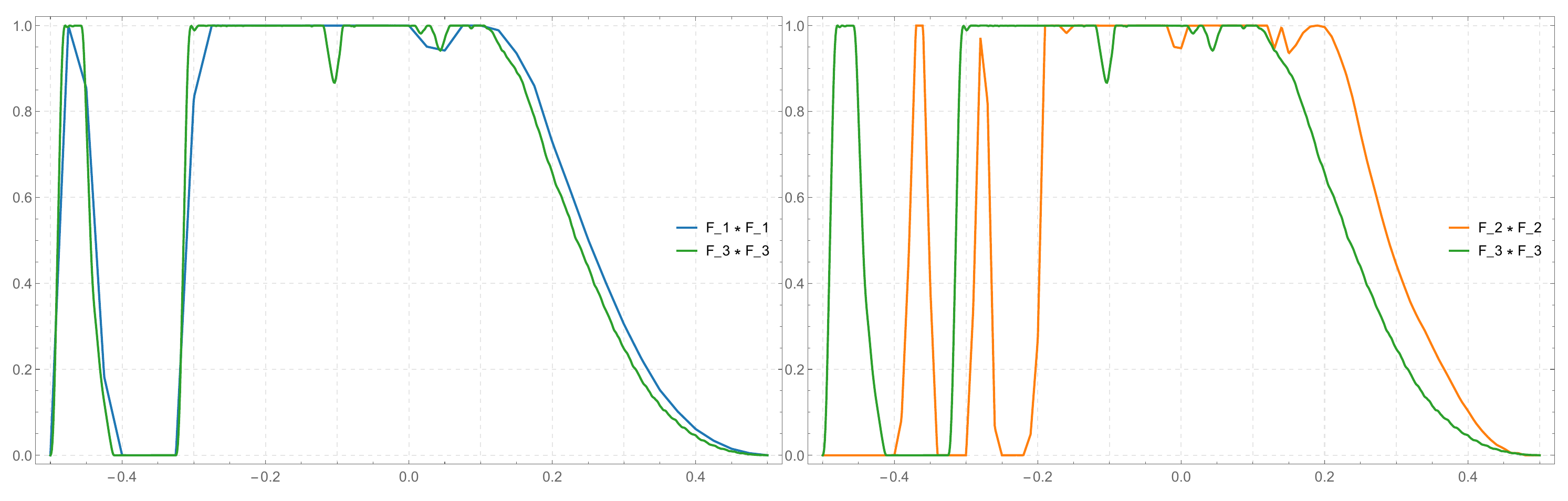}
\caption{Comparing $F_1 \ast F_1$ with $F_3 \ast F_3$ (left) and $F_2 \ast F_2$ with $F_3 \ast F_3$ (right)}
\end{figure}

Graphically, Matolcsi and Vinuesa's $F_1 \ast F_1$ is quite similar to our $F_3 \ast F_3$
while AlphaEvolve's $F_2 \ast F_2$ is somewhat different. One way to quantify this
is to measure the correlation/angle between two (continuous) functions by using the standard inner product
$\langle f, g \rangle = \int_{-1/2}^{1/2}f(x)g(x)\, dx$. 
Numerical computations give that
\begin{align*}
\frac{\int_{-1/2}^{1/2}(F_{1} \ast F_1)(x)(F_3 \ast F_3)(x)\, dx}{(\int_{-1/2}^{1/2}(F_{1} \ast F_1)(x)^{2}\, dx)^{1/2}(\int_{-1/2}^{1/2}(F_{3} \ast F_{3})(x)^{2}\, dx)^{1/2}} \approx 0.996342
\end{align*}
while
\begin{align*}
\frac{\int_{-1/2}^{1/2}(F_{2} \ast F_2)(x)(F_3 \ast F_3)(x)\, dx}{(\int_{-1/2}^{1/2}(F_{2} \ast F_2)(x)^{2}\, dx)^{1/2}(\int_{-1/2}^{1/2}(F_{3} \ast F_{3})(x)^{2}\, dx)^{1/2}} \approx 0.848466
\end{align*}
with the ratios being equal to 1 if the functions are the same.
Note that despite $F_1$ and $F_3$ (and their associated autoconvolutions) are quite similar,
we found $F_3$ independent from $F_1$, that is, our gradient based search did not presuppose knowledge of $F_1$.

\section{Description of algorithm}\label{desc}
We now describe how we found our step function in Theorem \ref{main}.
Given a large integer $N$, our program is to compute:
\begin{align*}
\argmax_{\mb{v} = (v_0, \ldots, v_{N - 1}) \in \R^{N}_{\geq 0}}Q_{N}(\mb{v}) \coloneqq \argmax_{\mb{v} = (v_0, \ldots, v_{N - 1}) \in \R^{N}_{\geq 0}}\frac{\frac{1}{3}\sum_{j = 0}^{2N - 1}L_{j}^{2} + L_{j}L_{j + 1} + L_{j + 1}^{2}}{(\max_{0 \leq j \leq 2N - 1}L_j)(\frac{1}{2}\sum_{j = 0}^{2N - 1}L_{j} + L_{j + 1})}
\end{align*}
where
$$L_j = \sum_{n = \max(0, j - N)}^{\min(j - 1, N - 1)}v_{n}v_{j - n - 1}.$$
Given a vector $\mb{v}$, by replacing the $j$th component
$\mb{v}_j$ with $\mb{v}_{j}/\sum_{\ell = 1}^{N}\mb{v}_\ell$, the above
optimization problem is equivalent to computing
$\argmax_{\mb{v} \in [0, 1]^{N}}Q_{N}(\mb{v})$.
By compactness of the $N$ dimensional unit hypercube and continuity
of $Q_{N}(\mb{v})$, a maximum must exist (though of course
this maximum can depend on $N$). We shall think
of $Q_{N}(\mb{v})$ as the score of a given vector $\mb{v}$ in the unit hypercube.

We briefly give a high level overview of our computational strategy.
The initial optimization phase uses simulated annealing:
\begin{enumerate}
\item We first choose $N = 23$.
\item We implement a coarse search to get the big picture
idea of the peaks of $Q_{N}(\mb{v})$.
We first initialize a random vector $\mb{V}^{(0)}$ in the unit
hypercube. We call it $\mb{V}_{best}$.

We then start with a perturbation scale parameter $S^{(0)} = 25$
and define $S^{(k+1)} \coloneqq S^{(k)} \cdot 0.999$.
We then iterate $\sim 1000$ times as follows: at iteration $k$,
we generate a fixed batch of candidate vectors as follows:
$$(\mb{V}_{candidate})_{j} \coloneqq |(\mb{V}_{best})_{j} + S^{(k)} \cdot \xi_{j}|$$
where $\{\xi_{j}\}_{j = 1}^{N}$ are i.i.d. uniform random variables on $(-1/2, 1/2)$. Of this batch of random vectors we compute
each $Q_{N}(\mb{V}_{candidate})$ and choose the vector (after normalizing
to lie within the unit hypercube) with the highest score to be the
new $\mb{V}_{best}$. The iteration then proceeds with iteration $k + 1$.

\item With this coarse search complete, we now fine tune the search
and run the same process as above except starting with the 
$\mb{V}_{best}$ from (1) and now $S^{(0)} = 0.05$ and define
$S^{(k + 1)} \coloneqq S^{(k)} \cdot 0.99998$ and iterate the process
outlined in (1) a total of $\sim 1000$ times.

\item We then repeat (3) a total of $\sim 400$ times.
\end{enumerate}

Next with this $\mb{V}_{best} \in \R^{23}$ from the first optimization
phase, we upscale this vector to one in $\R^{115}$ by 
repeating each entry of $\mb{V}_{best}$ five times. 
We call this upscaled vector $\mb{V}$.
We now apply the following gradient based approach:
\begin{enumerate}
\item For an integer $s \in [0, 10^{6} - 1]$, we will output
$10^6$ many vectors $\mb{V}^{(s)}$ and their scores.
\item Fix an $s$. At step $0$, let $\mb{V}^{(s)}$ be initially defined to be $\mb{V}$. We then iterate $\sim 1000$ times the following: at step $i$,
create (a discretized version of) the vector
$$\mb{V}^{(s)} + \frac{0.01}{(i + 1)^{1/4} \cdot (s + 1)^{1/3}}(1 - Q_{N}(\mb{V}^{(s)}))\cdot \frac{(\nabla Q_N)(\mb{V}^{(s)})}{\|(\nabla Q_N)(\mb{V}^{(s)})\|}.$$
After zeroing out any negative component, we let this be the
new $\mb{V}^{(s)}$ vector and proceed to step $i + 1$.
After $\sim 1000$ steps, we define $\mb{V}$ to be the final $\mathbf{V}^{(s)}$ and output its score.
Proceed to $s + 1$.
\end{enumerate}
Finally, with this vector $\mathbf{V}$ in $\R^{115}$, we then apply the initial optimization
phase with $N = 115$ and then upscale it to one in $\R^{575}$ and apply the above gradient
based approach.

The Python code for the initial optimization phase can
be found in \texttt{Inital\_HPC\_code.py} of our GitHub repository,
the code for the second more refined optimization phase can be found
in \texttt{Refine\_HPC\_code.py} of our repository. 
This yielded a list of coefficients in \texttt{f95.py}.
We then multiplied all entries by $2^{24}$ and rounded to the nearest integer,
giving the coefficients listed in the next section. This list of integer
coefficients can be found in \texttt{coeffBL.txt}.

\appendix \section{Step function coefficients}\label{coeff}
Below we list the the heights $\{v_n\}$ of the step function $f = \sum_{n = 0}^{574}v_n 1_{[n, n + 1)}$.
The table is to be read by row, so for example, $v_0 = 25437, v_1 = 56281, v_3 = 89185$.
Given this list of coefficients, one can run
the following Mathematica code below to compute \eqref{obj} for this list. The code below takes in as input a list of coefficients \texttt{v} and outputs \eqref{obj} (note that
lists in Mathematica start with 1 rather than 0 which explains why the small
shift index between how \texttt{L[j\_]} is defined over \eqref{ljdef}, but we keep \eqref{ljdef} as arrays in Python start with 0).\\

\small
\begin{verbatim}
len := Length[v]
L[j_] := Sum[v[[n]]*v[[j - n + 1]], {n, Max[1, j - len + 1], Min[j, len]}]

SquareL2 := (1/3) Sum[L[j]^2 + L[j]*L[j + 1] + L[j + 1]^2, {j, 0, 2 len}]
LInfinity := Max[Table[L[j], {j, 0, 2 len - 1}]]
L1 := (1/2) Sum[L[j] + L[j + 1], {j, 0, 2 len - 1}]

SquareL2/(LInfinity*L1)
\end{verbatim}
\normalsize

For verification of the 0.901564, the following table with the above Mathematica code is listed as \texttt{coeffBL} under the \texttt{Verify.nb} file in our GitHub respository.
We have also included the list of coefficients from \cite{MatolcsiVinuesa} and \cite{AlphaEvolve} (listed as \texttt{coeffMV} and \texttt{coeffAE}, respectively).

\begin{longtable}{llllllllll}
25437 & 56281 & 89185 & 114544 & 132666 & 143018 & 151257 & 150303 & 146083 & 143407 \\
144055 & 144553 & 148756 & 148683 & 141713 & 130151 & 114887 & 102118 & 94662 & 83763 \\
72587 & 67127 & 69821 & 72861 & 69291 & 75656 & 75568 & 76758 & 74626 & 69954 \\
64348 & 64937 & 64283 & 60723 & 56221 & 56044 & 55763 & 53595 & 52222 & 53197 \\
56146 & 56654 & 56521 & 54045 & 52519 & 51844 & 51561 & 51055 & 46657 & 40285 \\
18427 & 7 & 3 & 2 & 1 & 1 & 1 & 0 & 0 & 1 \\
1 & 1 & 2 & 1 & 1 & 1 & 1 & 0 & 0 & 0 \\
0 & 0 & 0 & 0 & 0 & 0 & 0 & 0 & 0 & 1 \\
0 & 0 & 0 & 0 & 0 & 0 & 0 & 0 & 0 & 0 \\
0 & 0 & 0 & 0 & 0 & 0 & 0 & 0 & 0 & 0 \\
0 & 3 & 4 & 4 & 1 & 0 & 1 & 1 & 3 & 4 \\
3 & 0 & 0 & 0 & 0 & 0 & 0 & 0 & 0 & 0 \\
0 & 0 & 0 & 0 & 0 & 0 & 0 & 0 & 0 & 0 \\
0 & 0 & 0 & 0 & 0 & 0 & 0 & 0 & 0 & 0 \\
0 & 0 & 0 & 0 & 0 & 0 & 0 & 0 & 2 & 3 \\
3 & 3 & 2 & 1 & 0 & 0 & 1 & 2 & 3 & 3 \\
2 & 0 & 0 & 2 & 3 & 3 & 1 & 0 & 1 & 1 \\
2 & 2 & 2 & 1 & 0 & 0 & 0 & 0 & 0 & 0 \\
0 & 0 & 0 & 0 & 0 & 0 & 0 & 0 & 2 & 2 \\
2 & 1 & 2 & 4 & 3 & 6 & 10 & 11 & 13 & 15 \\
15 & 6721 & 21513 & 37220 & 49269 & 56347 & 64491 & 69554 & 70920 & 69479 \\
70548 & 69021 & 71077 & 75321 & 79184 & 75944 & 67793 & 60720 & 56225 & 52324 \\
46466 & 42680 & 39415 & 32543 & 30727 & 31454 & 32731 & 33403 & 34021 & 37747 \\
41677 & 41681 & 37113 & 33297 & 34035 & 30821 & 27309 & 25576 & 25806 & 26145 \\
24918 & 24865 & 26194 & 26158 & 27590 & 29533 & 28794 & 26877 & 27476 & 26889 \\
29976 & 32076 & 38223 & 41254 & 43040 & 41154 & 41601 & 45871 & 46228 & 43457 \\
44525 & 46051 & 45769 & 47050 & 45949 & 43269 & 40238 & 37992 & 38670 & 39365 \\
38648 & 37712 & 39444 & 37935 & 36593 & 37993 & 38419 & 41782 & 42390 & 40152 \\
39494 & 37680 & 36909 & 36885 & 34366 & 32542 & 33995 & 33894 & 33936 & 34892 \\
35078 & 35874 & 36010 & 35848 & 35930 & 36587 & 35236 & 35757 & 33505 & 31980 \\
36406 & 37844 & 39176 & 38173 & 37214 & 37049 & 40200 & 40118 & 41103 & 41593 \\
39386 & 39261 & 38212 & 37955 & 37967 & 36884 & 36529 & 36225 & 36229 & 39558 \\
39165 & 38812 & 37868 & 38918 & 38653 & 39327 & 40263 & 39749 & 38569 & 37243 \\
35469 & 36359 & 35468 & 36688 & 37008 & 35592 & 36950 & 37621 & 38281 & 38696 \\
38017 & 37233 & 37823 & 37710 & 37493 & 36280 & 34880 & 35004 & 36186 & 35325 \\
38053 & 39850 & 38523 & 37446 & 36824 & 38413 & 40317 & 40652 & 39097 & 37652 \\
35139 & 35630 & 37379 & 36794 & 37362 & 36586 & 37036 & 37661 & 37091 & 38185 \\
41113 & 41266 & 38602 & 35196 & 37973 & 38825 & 38318 & 37600 & 35856 & 35262 \\
35541 & 37127 & 37451 & 37901 & 39269 & 39123 & 37956 & 37481 & 38275 & 38337 \\
38926 & 37316 & 35536 & 36344 & 37186 & 35762 & 34678 & 36390 & 39861 & 39489 \\
37976 & 37993 & 37256 & 34149 & 31686 & 28184 & 26284 & 23759 & 20260 & 17994 \\
17922 & 20081 & 23011 & 21818 & 19102 & 15593 & 14422 & 18174 & 28363 & 28839 \\
25665 & 22609 & 21701 & 23408 & 27782 & 33818 & 32357 & 28872 & 26451 & 26338 \\
28590 & 29215 & 29202 & 29365 & 29708 & 28446 & 28654 & 19372 & 16151 & 20889 \\
19429 & 21454 & 20068 & 20941 & 20096 & 19885 & 20727 & 23618 & 25110 & 23676 \\
19805 & 18828 & 19362 & 19544 & 21184 & 21353 & 22263 & 21574 & 19487 & 19067 \\
20818 & 23477 & 24904 & 27780 & 27178 & 25104 & 23683 & 18249 & 11753 & 11327 \\
13209 & 18430 & 19898 & 21569 & 20576 & 23150 & 22319 & 20532 & 19963 & 21172 \\
22049 & 20468 & 18309 & 15607 & 14671 & 15336 & 15447 & 17042 & 17752 & 19939 \\
17660 & 19872 & 19966 & 19530 & 19607 & 19084 & 17490 & 17382 & 16513 & 16396 \\
13438 & 10475 & 11658 & 13811 & 14499 & 14478 & 13284 & 13235 & 13374 & 13121 \\
13835 & 15625 & 15049 & 13554 & 10612 & 8114 & 8210 & 8557 & 7328 & 9389 \\
10178 & 12878 & 13488 & 14334 & 12454 & 11952 & 13060 & 11236 & 9692 & 9627 \\
9461 & 8238 & 5772 & 5285 & 6940 & 10141 & 12261 & 12118 & 9918 & 11121 \\
11244 & 12346 & 12735 & 10263 & 6928 & 6882 & 6806 & 6190 & 5457 & 7734 \\
6624 & 2676 & 3936 & 8607 & 10354 & 9826 & 6927 & 3859 & 3465 & 6052 \\
6613 & 5861 & 3600 & 291 & 57 & 118 & 3675 & 4709 & 3470 & 5792 \\
5071 & 2908 & 3084 & 5231 & 8452 & \text{} & \text{} & \text{} & \text{} & \text{} \\
\end{longtable}

\bibliographystyle{amsplain}
\bibliography{autoconv}
\end{document}